\documentclass[12pt]{amsart}

\usepackage{verbatim, amssymb, enumitem}

\usepackage[dvipdf]{graphicx}

\usepackage[breaklinks=true,colorlinks=true,linkcolor=green,citecolor=red,urlcolor=blue]{hyperref}

\allowdisplaybreaks

\renewcommand \a{\alpha}
\renewcommand \b{\beta}

\newcommand \br{\mathbb{R}}

\newcommand \Ric{\operatorname{Ric}}

\newcommand \Span{\operatorname{Span}}

\renewcommand \P{\mathcal{P}}

\newcommand\kg{\mathfrak k}
\newcommand\g{\mathfrak g}

\newcommand\h{\mathfrak h}
\newcommand\z{\mathfrak z}
\newcommand\m{\mathfrak m}

\newcommand \su{\mathfrak{su}}

\newcommand \f{\mathfrak{f}}
\newcommand \p{\mathfrak{p}}

\newcommand \ad{\operatorname{ad}}
\newcommand \Ad{\operatorname{Ad}}

\newcommand \diag{\operatorname{diag}}

\newcommand \Id{\operatorname{Id}}

\DeclareMathOperator{\Lie}{Lie}

\newcommand \<{\langle}
\renewcommand \>{\rangle}
\newcommand \ip{\<\cdot,\cdot\>}

\newtheorem{theorem}{Theorem}
\newtheorem*{theorem*}{Theorem}
\newtheorem{corollary}{Corollary}
\newtheorem*{corollary*}{Corollary}
\newtheorem*{conj*}{Conjecture}
\newtheorem{lemma}{Lemma}
\newtheorem{proposition}{Proposition}
\newtheorem*{prop*}{Proposition}

\theoremstyle{definition}

\newtheorem*{definition*}{Definition}

\theoremstyle{remark}
\newtheorem{remark}{Remark}

\newtheorem{example}{Example}
\newtheorem*{notation*}{Notation}
\newtheorem*{algorithm*}{Algorithm}

\begin{document}

\title[On invariant Riemannian metrics on Ledger--Obata spaces]{On invariant Riemannian metrics on \\ Ledger--Obata spaces}

\author{Y.~Nikolayevsky}
\address{Department of Mathematics and Statistics, La Trobe University, Melbourne, Australia 3086}
\email{Y.Nikolayevsky@latrobe.edu.au}
\thanks{The first named author was partially supported by ARC Discovery Grant DP130103485. }

\author{Yu.G.~Nikonorov}
\address{Southern Mathematical Institute of Vladikavkaz Scientific Centre of the Russian Academy of Sciences, Markus st. 22, Vladikavkaz, 362027, Russia}
\email{nikonorov2006@mail.ru}
\thanks{The second named author was partially supported by Grant 1452/GF4 of Ministry of Education and Sciences of the Republic of Kazakhstan for 2015-2017.}




\begin{abstract}
We study invariant metrics on Ledger--Obata spaces $F^m/\diag(F)$. We give the classification and an explicit construction of all naturally reductive metrics,
and also show that in the case $m=3$, any invariant metric is naturally reductive. We prove that a Ledger--Obata space is a geodesic orbit space if
and only if the metric is naturally reductive. We then show that a Ledger--Obata space is reducible if and only if it is isometric to the product of
Ledger--Obata spaces (and give an effective method of recognising reducible metrics), and that the full connected isometry group of an irreducible
Ledger--Obata space $F^m/\diag(F)$ is $F^m$. We deduce that a Ledger--Obata space is a geodesic orbit manifold if and only if it is the product of
naturally reductive Ledger--Obata spaces.
\end{abstract}

\maketitle

\section{Introduction}
\label{s:intro}


A Riemannian manifold $(M,g)$ is called a \emph{geodesic orbit manifold} (or a manifold with homogeneous geodesics, or a GO-manifold) if any
geodesic of $M$ is an orbit of a $1$-parameter subgroup of the full isometry group of $(M,g)$ (one loses no generality by replacing the full
isometry group by its connected identity component). A Riemannian manifold $(M=G/H,g)$, where $H$ is a compact subgroup of the Lie group $G$ and $g$
is a $G$-invariant Riemannian metric on $M$, is called a \emph{geodesic orbit space} (or a space with homogeneous geodesics, or a GO-space) if any
geodesic of $M$ is an orbit of a $1$-parameter subgroup of the group $G$. Hence a Riemannian manifold $(M,g)$ is a geodesic orbit manifold, if it is
a geodesic orbit space with respect to its full isometry group. This terminology was introduced in \cite{KV} by O.~Kowalski and L.~Vanhecke who initiated
the systematic study of such spaces.

The class of geodesic orbit spaces includes (but is not limited to) symmetric spaces, weekly symmetric spaces, normal and generalized normal
homogeneous spaces and naturally reductive spaces. For the current state of knowledge in the theory of geodesic orbit spaces and manifolds we
refer the reader to \cite{Nik2017,Arv} and the bibliographies therein.
It should be noted also some natural generalizations of geodesic orbit spaces from the mechanical point of view \cite{Toth}.

Let $(M=G/H, g)$ be a homogeneous Riemannian space and let $\g=\h\oplus \p$ be an $\Ad(H)$-invariant decomposition,
 where $\g=\Lie(G), \;\h=\Lie(H)$ and $\p$ is identified with the tangent space to $M$ at $eH$. The Riemannian metric $g$ is $G$-invariant and
 is determined by an $\Ad(H)$-invariant inner product $(\cdot,\cdot)$ on $\p$. The metric $g$ is called \emph{naturally reductive} if an
 $\Ad(H)$-invariant complement $\p$ can be chosen in such a way that $([X,Y]_{\p},X) = 0$ for all $X,Y \in \p$, where the subscript $\p$ denotes the
 $\p$-component. If this is the case, we say that the (naturally reductive) metric $g$ \emph{is generated by the pair} $(\p, (\cdot ,\cdot ))$.
 For comparison, on the Lie algebra level, $g$ is geodesic orbit if and only if for any $X \in \p$ (with any choice of $\p$), there
 exists $Z \in \h$ such that $([X+Z,Y]_{\p},X) =0$ for all $Y \in \p$ \cite[Proposition~2.1]{KV}. It immediately follows that any naturally
 reductive space is a geodesic orbit space; the converse is false when $\dim M \ge 6$.

\smallskip

Let $F$ is a connected, compact, simple Lie group. The Riemannian homogeneous space $F^m/\diag(F)$, where $F^m=F\times F\times\dots\times F, \; m\geq 2$,
and $\diag(F)=\{(x,x,\dots,x) \, | \, x\in F\}$, is called a \emph{Ledger--Obata space}. Ledger--Obata spaces were first introduced in \cite{LO1968} as
a natural generalization
of symmetric spaces $F^2/\diag(F)$, and have been actively studied since then. In particular, Einstein metrics on $F^3/\diag(F), F^4/\diag(F)$ and on
a general Ledger--Obata space were classified in \cite{Nik2002, CNN2017}.

Denote $\f$ the Lie algebra of the Lie group $F$, $\g = m\f=\f \oplus \f \oplus \dots \oplus \f$ the Lie algebra of $G=F^m$, and
$\h=\diag(\f)= \{(X,X,\dots,X)\, | \, X\in \f\} \subset \g$ the Lie algebra of $H=\diag(F)\subset G=F^m$. Let $\ip$ be minus the Killing form
on $\g$, and $\ip_i$ its restriction to the $i$-th summand $\f$.

We prove the following characterization of naturally reductive metrics on $F^m/\diag(F)$.
\begin{theorem} \label{th:natred}
Let $F$ be a connected, compact, simple Lie group. An invariant metric on a Ledger--Obata space $F^m/\diag(F)$ is naturally reductive if and only if
it is generated by a pair $(\p, (\cdot ,\cdot ))$ such that
\begin{enumerate} [label=\emph{(\alph*)},ref=\alph*]
  \item \label{it:natredideal}
  either $\p=(m-1)\f$ is an ideal in $\g$ and $(\cdot ,\cdot )$ is an $\ad(\p)$-invariant inner product on it, so that
  $(\cdot ,\cdot )=\sum_{i=1}^{m-1} \beta_i \ip_i$, where $\beta_i > 0$.

  \item \label{it:natredQ}
  or $\p$ is the orthogonal complement to $\h=\diag(\f)$ relative to an $\ad(\g)$-invariant quadratic form
  $Q=\sum_{i=1}^{m} \alpha_i \ip_i$ on $\g$ and $(\cdot ,\cdot ) = Q|_\p$, where
  \begin{enumerate}[label=\emph{(\roman*)},ref=\roman*]
    \item \label{it:natredQ+}
    either $\alpha_i > 0$ for all $i=1, \dots, m$,

    \item \label{it:natredQ-}
    or there exists $j=1, \dots, m$ such that $\alpha_j < 0$ and $\alpha_i > 0$ for all $i \ne j$, and $S:=\sum_{i=1}^m \alpha_i < 0$.
  \end{enumerate}
\end{enumerate}
\end{theorem}

Note that all the metrics from \eqref{it:natredideal} are reducible when $m >2$; this is not necessarily true for the metrics in \eqref{it:natredQ}. 

In Corollary~\ref{cor:natredgm} in Section~\ref{s:natred}, we give the uniform description of all inner products on the ideal $(m-1)\f$
(viewed as an $\Ad(H)$-invariant complement to $\h$ in $\g$) which produce a naturally reductive metric on $F^m/\diag(F)$. Using that description we
obtain the following.

\begin{proposition}[{\cite[Proposition~3]{CN2017}}]\label{ledob6.prop}
Every invariant Riemannian metric on a Ledger--Obata space $F^3/\diag(F)$, where $F$ is a connected, compact, simple Lie group, is naturally
reductive and hence is geodesic orbit.
\end{proposition}

The claim is no longer true for $F^m/\diag(F)$ with $m > 3$, by the parameter count.

\begin{theorem}\label{th:gointro}
A Ledger--Obata space with an invariant Riemannian metric is a geodesic orbit space if and only if it is naturally reductive.
\end{theorem}

Next we address the question of characterizing Ledger--Obata spaces which are geodesic orbit \emph{manifolds}. We prove the following
decomposition theorem which also gives the description of the full connected isometry group of a Ledger--Obata space.

\begin{theorem} \label{th:red}
Let $F$ be a connected, compact, simple Lie group. A Ledger--Obata space $F^m/\diag(F)$ with an invariant metric is isometric to the product of
Ledger--Obata spaces $F^{m_i}/\diag(F), \; i=1, \dots, s$, with invariant metrics which are irreducible \emph{(}as Riemannian manifolds\emph{)},
where $s \ge 1, \; m_i > 1$, and $\sum_{i=1}^s (m_i-1)=m-1$; such a decomposition is unique, up to relabelling. Furthermore, the full connected
isometry group of $F^m/\diag(F)$ is the direct product $\prod_{i=1}^s F^{m_i}$, where each factor $F^{m_i}$ acts from the left on the corresponding
factor $F^{m_i}/\diag(F)$ and trivially on the other factors.
\end{theorem}
It follows that the full connected isometry group of an invariant metric on a Ledger--Obata space $F^m/\diag(F)$ is $F^k$, where $m \le k \le 2(m-1)$.
The cases $k = m$ and $k = 2(m-1)$ correspond to irreducible and bi-invariant metrics respectively. For a special class of metrics on Ledger--Obata spaces,
the full connected isometry group was computed in \cite{Kow77}.

From Theorem~\ref{th:red} and Theorem~\ref{th:gointro} we obtain the following Corollary.
\begin{corollary} \label{cor:gom}
Let $F$ be a connected, compact, simple Lie group. A Ledger--Obata space $F^m/\diag(F)$ is a geodesic orbit manifold if and only if each of the
factors $F^{m_i}/\diag(F)$ in its irreducible decomposition is naturally reductive \emph{(}so that the metric on each of them has the form given
in Theorem~\ref{th:natred}\emph{)}. 
\end{corollary}

We add some remarks on invariant Einstein metrics on Ledger--Obata spaces. Recall that a Riemannian metric on a Ledger--Obata space $F^m/\diag(F)$ is called {\it standard} or {\it Killing} if it is induced by the minus Killing form of the Lie algebra of $F^m$. The Ricci curvature of the standard Riemannian metric $\rho_{st}$ on every Ledger--Obata space satisfies the equality $\Ric(\rho_{st})=\frac{m+2}{4m} \rho_{st}$, i.~e. $\rho_{st}$ is Einstein with the Einstein constant $\frac{m+2}{4m}$, see \cite{Nik2002} or \cite{CNN2017} for details. By Theorem 5.2 of Chapter~X of~\cite{KN}, we know that the Riemannian manifold $(F^m/\diag(F),\rho_{st})$ is irreducible. If (in the notation of Theorem~\ref{th:red}) an invariant Riemannian metric $\rho$ on $F^m/\diag(F)$ is the direct product of the multiples of the standard metrics on the factors $F^{m_i}/\diag(F)$, $i=1, \dots, s$, with the same Einstein constants, then $\rho$ is also Einstein. Such Einstein metrics are called {\it  routine} in \cite{CNN2017}. It is known that all Einstein invariant metrics on Ledger--Obata spaces are routine for $m \leq 4$. The questions about hypothetical non-routine Einstein metrics and geodesic orbit non-routine Einstein metrics on Ledger--Obata spaces were posed in \cite{CNN2017}. From Corollary~\ref{cor:gom}, Theorem~\ref{th:natred}, and \cite[Proposition 7.89]{Bes} we easily get that every geodesic orbit Einstein metric on any Ledger--Obata space is routine.

\smallskip
The paper is organised as follows. In Section~\ref{s:pre}, we give the necessary background. In Section~\ref{s:natred} we study naturally reductive
metrics on Ledger--Obata spaces; we prove Theorem~\ref{th:natred}, Corollary~\ref{cor:natredgm} and Proposition~\ref{ledob6.prop}. Geodesic orbit
Ledger--Obata spaces are studied in Section~\ref{s:go-legob}; we prove Theorem~\ref{theoclass.1} (of which Theorem~\ref{th:gointro} is a particular case)
using the so called super-adapted systems. In the first part of Section~\ref{s:LOm}, we find the full connected isometry group of a Ledger--Obata space and
prove Theorem~\ref{th:red}; then Corollary~\ref{cor:gom} follows. In the second part, we compute the holonomy of an invariant metric on a Ledger--Obata space
and give an  effective method of recognizing whether a given inner product produces a reducible metric.

\section{Preliminaries}
\label{s:pre}


\subsection{Geodesic orbit condition}
\label{ss:gocondition}

Consider a compact homogeneous space $G/H$ with a semisimple group $G$ and denote by $\ip$ the minus Killing form of the Lie algebra $\g=\Lie(G)$.
Let $\m$ be the $\ip$-orthogonal complement to $\h=\Lie(H)$ in $\g$. Any invariant Riemannian metric $g$ on $G/H$ is determined by a suitable $\Ad(H)$-invariant
inner product $(\cdot,\cdot)$ on $\m$ and vise versa. For any such metric we can canonically define \emph{the metric endomorphism} $A:\m \rightarrow \m$ such
that $(X,Y)=\<A X, Y \>$, for $X,Y \in \m$. The endomorphism $A$ is $\Ad(H)$-equivariant, positive definite and symmetric with respect to $\ip$.
We have the following useful lemma.

\begin{lemma}[{\cite[Proposition~1]{AA}}]\label{GO-criterion}
A compact homogeneous Riemannian space $(G/H,g)$ with the metric endomorphism $A$ is a geodesic orbit space if and only if for any $X \in \m$ there exists $Z \in \h$ such that $[X+Z, AX]\in \h$.
\end{lemma}

For any eigenvalue $\alpha$ of the metric endomorphism $A$, the corresponding eigenspace $\m_\alpha$ is $\Ad(H)$-invariant; the eigenspaces $\m_\a, \m_\beta, \; \a \ne \b$, are pairwise orthogonal with respect to both $(\cdot,\cdot)$ and $\ip$. Denote $\z_\h(X)$ the centraliser of $X \in \g$ in $\h$.

\begin{lemma}[{\cite[Proposition~5]{AN}, \cite[Section~5]{Nik2017}}] \label{l:brackets}
Suppose a compact homogeneous Riemannian space $(G/H,g)$ with the metric endomorphism $A$ is geodesic orbit. Then
\begin{enumerate} [label=\emph{(\roman*)}, ref=\roman*]

  \item \label{it:alphabeta}
  For any $X \in \m_\alpha, \; Y \in \m_\beta$ with $\alpha \ne \beta$, there exists $Z \in \h$ such that $[X,Y]=\frac{\alpha}{\beta-\alpha} [Z,X]+\frac{\beta}{\beta-\alpha} [Z,Y]$.

  \item \label{it:centralisers}
  For any $X \in \m_\alpha, \; Y \in \m_\beta$ with $\alpha \ne \beta$, there exist $Z_1 \in \z_\h(X), \; Z_2 \in \z_\h(Y)$ such that $[X,Y]= [Z_2,X] + [Z_1,Y]$.

  \item \label{it:twoinone}
  If $X, Y  \in \m_\alpha$ are such that $[\h,X] \perp Y$, then $[X, Y] \in \m_\alpha$.
\end{enumerate}
\end{lemma}

\subsection{Linear holonomy algebra}
\label{ss:holo}

Let $G/H$ be a compact homogeneous space with a semisimple group $G$. In the notation of Section~\ref{ss:gocondition}, extend the metric endomorphism $A$ to the operator $C$ on $\g$ which is symmetric relative to $\ip$ and is defined by $C|_\m=A$ and $C|_\h=0$. For $Z \in \g$ let $D_Z:\m \rightarrow \m$ be defined by $D_Z (Y)= [Z, Y]_{\m}$, for $Y \in \m$, where the subscript $\m$ denotes the $\m$-component.

\begin{lemma}[{\cite[Theorem~2.3]{K}}]\label{reduce-criterion}
The linear holonomy algebra of $G/H$ is the Lie algebra generated by all the operators on $\m$ of the form
\begin{equation*}
\Gamma_Z:=D_Z+C^{-1}D_Z C- C^{-1}D_{CZ},
\end{equation*}
where $Z \in \g$.
\end{lemma}
Note that for $Z\in \h$, we get $\Gamma_Z=2 D_Z=2 \ad_Z|_\m$. For any $X \in \m_\a, \; Y \in \m_\b$ and $Z \in \m_{\gamma}$ we obtain
\begin{align*}
(\Gamma_Z X,Y) &=\< C (\Gamma_Z X),Y \>=\< C[Z,X],Y \>+\< [Z,CX],Y\> -\< [CZ,X],Y\> \\
&=(\alpha+\beta-\gamma)\< [Z,X],Y \>=(1+(\alpha-\gamma)\beta^{-1})([Z,X],Y).
\end{align*}
This gives a comprehensive description of the linear holonomy algebra. Note that although the subspaces of $\m$ invariant under the linear holonomy algebra are $\ad(\h)$-modules, they are not necessarily the sums of the eigenspaces $\m_\a$ (more precisely, they may not be $A$-invariant). However we have the following Lemma.

\begin{lemma}[{\cite[Lemma~3.4A]{K}}]\label{l:holinv}
If a subspace $\m_1 \subset \m$ is invariant under the linear holonomy algebra, and $\m_2$ is the $(\cdot,\cdot)$-orthogonal complement to $\m_1$ in $\m$, then both $\g_1=\h \oplus \m_1$ and $\g_2=\h \oplus \m_2$ are subalgebras in $\g$.
\end{lemma}

The formula for the full linear holonomy group is given in \cite[Theorem~3.3]{K}.

\subsection{Invariant metrics on Ledger--Obata spaces}
\label{ss:metrics}

Let $F$ be a connected, compact, simple Lie group, $G=F^m$ and $H=\diag(F)\subset G=F^m$. Let $\f = \Lie(F)$ and $\h = \diag(\f) = \{(X,X,\dots,X)\,|\, X\in \f\} = \Lie(H)$. The space of all invariant Riemannian metrics on the space $G/H=F^m/\diag(F)$ is naturally identified with the set of all $\Ad(H)$-invariant inner products on any $\Ad(H)$-invariant complement to $\h$ in $\g$.

For $a=(a_1, a_2, \dots a_m) \in \br^m$ and $X \in \f$, denote $a \otimes X$ the vector $(a_1 X, a_2 X, \dots a_m X) \in \g$. Denote $\ip$ minus the Killing form on $\g=m\f$, and $\ip_i$, its restriction to the $i$-th copy of $\f$ in $m\f$.

From the following lemma we see that there are many choices of an $\Ad({H})$-invariant complement to $\h$ in $\g=m \f$.

\begin{lemma}[{\cite[Lemma~1.1]{Nik2002}}]\label{strinmod}
Every irreducible $\Ad({H})$-invariant submodule of $\g$ has the form
\begin{equation}\label{irred}
\left\{(\alpha_1, \alpha_2,\dots, \alpha_{m}) \otimes X \subset \g\,|\, X \in \f \right\}
\end{equation}
for some fixed $\alpha_i \in \mathbb{R}$.
\end{lemma}

In particular, $\h=\diag(\f)=\{(X,X,\dots,X)\,|\, X \in \f\}$ is one such submodule. We can choose arbitrary $m-1$ linear independent vectors $(\alpha_1, \alpha_2, \dots, \alpha_{m}) \in \br^m$ whose span does not contain the vector $(1,1,\dots, 1)$ and then construct the $\Ad({H})$-invariant complement to $\h$ in $\g$ as the direct sum of the corresponding $m-1$ irreducible submodules defined by \eqref{irred}. Throughout the paper we will mostly use the following two choices. 

Let $\g_k, \; k=1,\dots,m$, be the ideal of $m \f$ of the form $\f\oplus\dots \oplus \f\oplus 0 \oplus \f\oplus\dots \oplus \f$, where $0$ replaces the $k$-th copy of $\f$ in $m \f$. Any of $\g_k$ can be chosen as an $\Ad(H)$-invariant complement to $\h$ in $\g$; we will often use the submodule $\p=\g_{m}$. Any inner product $(\cdot, \cdot)$ on $\p$ is determined by
\begin{equation} \label{eq:ineerp}
\bigl((x_1,x_2,\dots,x_{m-1},0) \otimes Z, (x_1,x_2,\dots,x_{m-1},0) \otimes Z \bigr)=f(x_1,x_2,\dots,x_{m-1})\cdot \langle Z, Z \rangle,
\end{equation}
for any $x_1,\dots, x_{m-1} \in \mathbb{R}, \; Z \in \f$, where $f(x_1,x_2,\dots,x_{m-1})=\sum_{i,j=1}^{m-1} a_{ij}\, x_i x_j, \; a_{ij}\in \mathbb{R}$,  $a_{ij} = a_{ji}$, is a positive definite quadratic form.

Another natural choice of an $\Ad(H)$-invariant complement is the $\ip$-orthogonal complement to $\h$ in $\g$ which we already met before is defined by
\begin{equation} \label{eq:defm}
\m=\Big\{\Span \bigl((\alpha_1, \alpha_2,\dots, \alpha_{m}) \otimes X\bigr)\subset \g \, | \, X \in \f,\; \sum\nolimits_{i=1}^m \alpha_i =0 \Big\}.
\end{equation}


We introduce the notation which will be used in the proof of the following Lemma and also further in Section~\ref{s:go-legob}. For vectors  $a=(a_1,a_2,\dots,a_n), \; b=(b_1,b_2,\dots,b_n)  \in \mathbb{R}^n$, denote $a\diamond b$ the vector $c=(c_1,c_2,\dots,c_n)\in \mathbb{R}^n$ such that $c_i=a_ib_i$, for $i=1,\dots,n$.

\begin{lemma} \label{l:subal}
Let $\f$ be a simple, compact Lie algebra and $n \in \mathbb{N}$. Suppose a Lie subalgebra $\kg \subset n \f$ contains the subalgebra $\diag(\f) \subset n \f$. Then there is a decomposition $n \f = \oplus_{i=1}^s n_i \f$ into the direct sum of ideals \emph{(}where $s \ge 1, \; n_i \ge 1$ and $\sum_{i=1}^s n_i = n$\emph{)} such that $\kg = \kg_1 \oplus \kg_2 \oplus \cdots \oplus \kg_s$, where $\kg_i=\diag(\f) \subset n_i \f$.
\end{lemma}

\begin{proof}
The subalgebra $\kg$ is an $\ad(\diag(\f))$-submodule of $n \f$, and so by Lemma~\ref{strinmod}, there is a linear subspace $U \subset \mathbb{R}^{n}$ such that $\{(\alpha_1, \alpha_2,\dots, \alpha_{n}) \otimes X \,|\, X \in \f\} \subset \kg$ if and only if $\alpha =(\alpha_1, \alpha_2, \cdots, \alpha_{n})\in U$. Since $\kg$ is a subalgebra, we have $\alpha \diamond \beta \in U$, for every $\alpha,\beta \in U$. For $l=1, \dots, n$, denote $t_l=(1, \dots, 1, 0, \dots, 0)$ the vector in $\br^n$ whose first $l$ entries are ones and the remaining $n-l$ are zeros. By assumption, $t_n \subset U$. If $U = \br t_n$, there is nothing to prove. Otherwise, there exist non-zero vectors in $U$ some of whose entries are zeros. Let $\alpha \in U$ be a non-zero vector with the maximal number of zero entries. Up to relabelling we may assume that $\alpha_i \neq 0$ for $i\leq l \in \mathbb{N}$ and $\alpha_i=0$ for $i>l$. Then for all $\b \in U$ (including $\b=\a$) we have $\b_1=\b_2=\dots =\b_l$, as otherwise $\alpha \diamond \b -  {\b_1} \alpha \in U$ is non-zero and has more zero entries than $\a$. It follows that $U = \br t_l \oplus U'$, where $U'$ is a linear subspace in $\br^n$ consisting of vectors all of whose first $l$ entries are zeros, which is closed under $\diamond$ and which contains the vector $t_n-t_l$. Repeating the argument we obtain a sequence $l=l_1 < l_2 < \dots < l_s=n$ such that $U = \Span(t_{l_1}, t_{l_1}-t_{l_2}, \dots, t_{l_s}-t_{l_{s-1}})$, and the claim follows if we define $n_1=l_1$ and $n_i=l_i-l_{i-1}$ for $i=2, \dots, s$.
\end{proof}

\section{Naturally reductive metrics on Ledger--Obata spaces}
\label{s:natred}

In this section, we prove Theorem~\ref{th:natred} and Proposition~\ref{ledob6.prop} and also find all inner products on $\g_m$ which produce naturally reductive metrics on the Ledger--Obata space $F^m/\diag(F)$.

\begin{proof}[Proof of Theorem~\ref{th:natred}]
Suppose $\p$ is an $\ad(\h)$-invariant complement to $\h$ in $\g$. Then the space $\mathfrak{q}=[\p,\p]+\p$ is an ideal in $\g$.

By \cite[Theorem~4]{K56}, if a naturally reductive metric is generated by a pair $(\p, (\cdot ,\cdot ))$, then there is a (unique) $\ad(\mathfrak{q})$-invariant non-degenerate quadratic form $Q$ on $\mathfrak{q}$ such that
\begin{equation}\label{eq:Kostant}
    Q(\p, \mathfrak{q} \cap \h)=0, \qquad \text{and} \qquad Q|_{\p}=(\cdot ,\cdot ).
\end{equation}
The converse is also true: if $Q$ is an $\ad(\mathfrak{q})$-invariant non-degenerate quadratic form which satisfies the first equation of \eqref{eq:Kostant} whose restriction to $\p$ is positive definite, then that restriction defines a naturally reductive metric; this follows from $\ad(\mathfrak{q})$-invariancy of $Q$ and from the fact that $\mathfrak{q}$ is complemented in $\g$ by an ideal.

By the dimension count, we have two cases: $\mathfrak{q}=\p$ and $\mathfrak{q}=\g$. In the first case, $\p$ by itself is an ideal, so $\p=(m-1)\f$ and $Q=
(\cdot ,\cdot )$ has the form given in \eqref{it:natredideal}. In the second case, $Q$ is $\ad(\g)$-invariant, so $Q=\sum_{i=1}^{m} \alpha_i \ip_i$, with $\alpha_i \ne 0$. Then by \eqref{eq:Kostant} the space $\p$ is the direct sum of the $\h$-modules $\p_j=\{(t_1^j, t_2^j, \dots, t_m^j) \otimes X \, | \, X \in \f\}, \; j=1, \dots, m-1$, such that $\sum_{i=1}^m \alpha_i t_i^j=0$ and that the rank of the $(m-1) \times m$ matrix $T=(t_i^j)$ is $m-1$. Then $(\cdot ,\cdot )$, the restriction of $Q$ to $\p$, is given by $((t_1^j, t_2^j, \dots, t_m^j) \otimes X, (t_1^k, t_2^k, \dots, t_m^k) \otimes Y)=\sum_{i=1}^m \alpha_i t_i^j t_i^k \<X, Y\>$, and so $(\cdot ,\cdot )$ is positive definite if and only if the $(m-1) \times (m-1)$-matrix $T \diag(\alpha_1, \dots, \alpha_m) T^t$ is. Let $\tilde{T}$ be the $m \times m$-matrix obtained by adding to $T$ an extra row of $m$ ones at the bottom. Then we have
\begin{equation*}
    \tilde{T} \diag(\alpha_1, \dots, \alpha_m) \tilde{T}^t= \left(
                                                               \begin{array}{cc}
                                                                 T \diag(\alpha_1, \dots, \alpha_m) T^t & 0 \\
                                                                 0 & \sum_{i=1}^m \alpha_i \\
                                                               \end{array}
                                                             \right).
\end{equation*}
and so by the law of inertia, $(\cdot ,\cdot )$ is positive definite if and only if either $\alpha_i > 0$, for all $i=1, \dots, m$, or exactly one of $\alpha_i$ is negative and $\sum_{i=1}^m \alpha_i < 0$.
\end{proof}

\begin{remark} \label{rem:normal}
Note that a naturally reductive metric on $F^m/\diag(F)$ is normal if and only if it corresponds to the pairs $(\p, (\cdot ,\cdot ))$ in \eqref{it:natredideal} and \eqref{it:natredQ}\eqref{it:natredQ+} of Theorem~\ref{th:natred}.
\end{remark}

We now consider the presentations of all the naturally reductive metrics on the same $\h$-submodule $\g_m=\f\oplus \f \oplus \dots \oplus \f \oplus 0$ complementary to $\h$. In the notation of Section~\ref{ss:metrics} we have the following.

\begin{corollary} \label{cor:natredgm}
Let $F$ be a connected, compact, simple Lie group and let an invariant metric on a Ledger--Obata space $F^m/\diag(F)$ be defined by an inner product $(\cdot, \cdot)$ on $\g_m$ for a positive definite quadratic form $f(x_1,x_2,\dots,x_{m-1})=\sum_{i,j=1}^{m-1} a_{ij}\, x_i x_j$ on $\br^{m-1}$. The metric is naturally reductive if and only if
\begin{enumerate}[label=\emph{(\alph*)},ref=\alph*]
  \item \label{it:natredgmgm}
  either $a_{ii} > 0$ and $a_{ij}=0$ for $i \ne j$;

  \item \label{it:natredgmgk}
  or there exists $k=1, \dots, m-1$ such that $-a_{ik}=a_{ii} > 0$ for $i \ne k$, $a_{ij}=0$ for $i \ne j \ne k \ne i$, and $a_{kk} > \sum_{i \ne k} a_{ii}$;
  \item \label{it:natredgmQ}
  or $a_{ij} = \delta_{ij} \alpha_i - \frac{\alpha_i\alpha_j}{S}$, where $\alpha_1, \dots, \alpha_m \in \br \setminus \{0\}, \; S=\sum_{i=1}^m\alpha_i$, and either $\alpha_i > 0$, for all $i=1, \dots, m$, or exactly one of $\alpha_i$ is negative and $S < 0$.
\end{enumerate}
\end{corollary}
\begin{proof}
The proof follows by a direct calculation from Theorem~\ref{th:natred}. The cases \eqref{it:natredgmgm} and \eqref{it:natredgmgk} correspond to the inner products in Theorem~\ref{th:natred}\eqref{it:natredideal} when the ideal $\p=(m-1)\f$ is equal or is not equal to $\g_m$ respectively. The case \eqref{it:natredgmQ} corresponds to the inner product in Theorem~\ref{th:natred}\eqref{it:natredQ}.
\end{proof}

Note that the conditions imposed on $\alpha_i$ in Corollary~\ref{cor:natredgm}\eqref{it:natredgmQ} and in Theorem~\ref{th:natred}\eqref{it:natredQ} are in essence equivalent to the fact that $(\cdot, \cdot)$ is positive definite.

\smallskip

For consistency, we include the proof of Proposition~\ref{ledob6.prop} (which is different from and is shorter than that in \cite{CN2017}).

\begin{proof}[Proof of Proposition~\ref{ledob6.prop}] Consider the quadratic form $f(x_1,x_2)=a_{11}x_1^2+2a_{12}x_1x_2+a_{22}x_2^2$ that defines the inner product $(\cdot,\cdot)$ on $\g_3 = \f\oplus \f\oplus 0$ by \eqref{eq:ineerp}.

First suppose that $a_{12}(a_{11}+a_{12})(a_{22}+a_{12})= 0$. Note that $a_{11}+a_{12}\geq 2\sqrt{a_{11}a_{12}}>2|a_{12}|$. Hence $a_{11}+a_{12}=0$ implies $a_{22}>a_{11}$ and $a_{22}+a_{12}=0$ implies $a_{22}<a_{11}$. Therefore, if $a_{12}=0$ or $a_{11}+a_{12}=0$ or $a_{22}+a_{12}=0$, then the metric $(\cdot,\cdot)$ is naturally reductive (and even normal homogeneous) by Corollary~\ref{cor:natredgm}\eqref{it:natredgmgm} and \eqref{it:natredgmgk}).

Now assume that $a_{12}(a_{11}+a_{12})(a_{22}+a_{12}) \neq 0$. By Corollary~\ref{cor:natredgm}\eqref{it:natredgmQ}, for the metric $(\cdot,\cdot)$ is naturally reductive if there exist $\alpha_1,\alpha_2, \alpha_3 \in \mathbb{R}$ such that
\begin{equation*}
a_{11}=\frac{\alpha_1(\alpha_2+\alpha_3)}{S}, \quad
a_{12}=\frac{-\alpha_1\alpha_2}{S}, \quad
a_{22}=\frac{\alpha_2(\alpha_1+\alpha_3)}{S}\, ,
\end{equation*}
where $S=\alpha_1+\alpha_2+\alpha_3$ and either $\alpha_1, \alpha_2, \alpha_3 > 0$, or exactly one of them is negative and $S < 0$. This system has the following solution: $(\alpha_1, \alpha_2,\alpha_3)=\left(\frac{D}{a_{22}+a_{12}}, -\frac{D}{a_{12}}, \frac{D}{a_{11}+a_{12}}\right)$, where $D=a_{11}a_{22}-a_{12}^2>0$, and we have $S=-\frac{D^2}{a_{12}(a_{11}+a_{12})(a_{22}+a_{12})}$. Now if $a_{12} > 0$, then $\alpha_1, \alpha_3 > 0 > \alpha_2, S$. If $a_{12} <0$, we can assume that $a_{11} \ge a_{22}$. Then by the argument above, $a_{11}+a_{12} > 0$ and then $\alpha_2, \alpha_3 > 0$ and $\alpha_1$ and $S$ have the same sign.
\end{proof}


\section{Geodesic orbit metrics on Ledger--Obata spaces}
\label{s:go-legob}

In this Section, it will be more convenient to choose and fix the $\ad (\h)$-invariant complement $\m$ to $\h=\diag(\f)$ in $\g=m\f$ as defined in \eqref{eq:defm}. Let $(\cdot,\cdot)$ be an $\ad (\h)$-invariant inner product on $\m$. We can decompose $\m$ into the direct sum of $\ad (\h)$-invariant, irreducible submodules $\p_1,\dots,\p_{m-1}$ which are orthogonal both relative to $(\cdot,\cdot)$ and to $\ip$. By Lemma~\ref{strinmod} we have
\begin{equation*}
\p_i=\{(b^i_1, b^i_2, \dots ,b^i_m) \otimes X\, | \, X\in \f\}, \quad i=1,\dots,m-1,
\end{equation*}
and so
\begin{equation}\label{basis_adapt}
\sum\limits_{k=1}^m b^i_k =0, \quad
\sum\limits_{k=1}^m (b^i_k)^2 =1, \quad
\sum\limits_{k=1}^m b^i_k b^j_k =0 \quad \mbox{for}\quad i\neq j,
\end{equation}
so that the vectors $b^i=(b^i_1,b^i_2,\dots,b^i_m), \; i=1,\dots,m-1$, constitute an orthonormal basis for $(1, \dots, 1)^\perp \subset \mathbb{R}^m$. We have \begin{equation}\label{decomploort}
(\cdot,\cdot)=\gamma_1 \ip|_{\p_1}+\gamma_2 \ip|_{\p_2}+ \dots +\gamma_{m-1} \ip|_{\p_{m-1}},
\end{equation}
for some positive numbers $\gamma_i$, and the metric endomorphism $A$ defined in Section~\ref{ss:gocondition} has the form
\begin{equation*}
A =\gamma_1 \Id|_{\p_1}+\gamma_2 \Id|_{\p_2}+\dots +\gamma_{m-1} \Id|_{\p_{m-1}}.
\end{equation*}

In a Euclidean space $\mathbb{R}^m$ with a fixed orthonormal basis, introduce the symmetric binary operation $\diamond$ by setting $a\diamond b=(a_1b_1, a_2b_2, \dots, a_m b_m)$, for $a=(a_1,a_2,\dots,a_m), b=(b_1,b_2,\dots,b_m) \in \br^m$ (as in Section~\ref{ss:metrics}). Then we have $[\p_i, \p_j]=\{(b^i \diamond b^j) \otimes X \, | \, X\in \f\}$. Suppose that $(\cdot,\cdot)$ is geodesic orbit. By Lemma~\ref{l:brackets}\eqref{it:alphabeta} we obtain
\begin{equation}\label{basis_sadapt}
b^i \diamond b^j \in \Span(b^i, b^j),
\end{equation}
for all $i, j$ such that $\gamma_i \neq \gamma_j$. For $\gamma \in \br$, let $V= \Span (b^i \, | \, \gamma_i=\gamma)$. Then by Lemma~\ref{l:brackets}\eqref{it:twoinone} we have
\begin{equation}\label{self}
a \diamond b \in V, \quad \text{for all orthogonal} \quad a, b \in V.
\end{equation}
We call a system of vectors $b^i=(b^i_1,b^i_2,\dots,b^i_m)\in \mathbb{R}^m, \; i=1,\dots,m-1$, \emph{adapted}, if it satisfies~\eqref{basis_adapt}. An adapted system is called \emph{super-adapted} if in addition it satisfies~\eqref{basis_sadapt}. A linear subspace $V \subset (1, 1, \dots, 1)^\perp \subset \mathbb{R}^m$ is called \emph{self-saturated} if it satisfies \eqref{self}.

From the above, every $\ad (\h)$-invariant inner product $(\cdot,\cdot)$ on $\m$ determines an adapted system of vectors $b^i=(b^i_1,b^i_2,\dots,b^i_m) \in \mathbb{R}^m, \; i=1,\dots,m-1$, via the decomposition \eqref{decomploort}. Moreover, if the eigenvalues $\gamma_i$ of $A$ are pairwise distinct, such a system is unique. If in addition $(\cdot,\cdot)$ is geodesic orbit, then this unique adapted system is super-adapted. We want to establish a similar correspondence in the case when the eigenvalues of $A$ are not necessarily simple.

\begin{lemma}\label{selfsatur2}
For any self-saturated linear subspace $V$ there is an orthonormal basis  $\{u_i\}$ such that $u_i \diamond u_j \in \mathbb{R} u_i \cup \mathbb{R} u_j$, for every $i \neq j$.
\end{lemma}

\begin{proof}
There is nothing to prove if $\dim V = 1$. Otherwise, choose a unit vector $v \in V$ which has the maximal number of zero coordinates and denote $V'=v^\perp \cap V$. For any $u \in V'$, the vector $u \diamond v \in V$ has at least as many zero coordinates as $v$ does; moreover, if $u \diamond v$ is not a multiple of $u$, then there is a non-zero linear combination of $v$ and $u \diamond v$ having more zero coordinates than $v$, in contradiction with the choice of $v$. It follows that $u \diamond v \in \mathbb{R} v$, for all $u \in V'$. Furthermore, let $u, w \in V'$ be orthogonal vectors. Then $u \diamond w \in V$ and $\<u \diamond w,v\> = \<u \diamond v,w\> = 0$, so that $u \diamond w \in V'$. It follows that $V'$ is again self-saturated and the claim follows by induction.
\end{proof}

This lemma easily implies the following.

\begin{proposition}\label{selfsatur3}
Every geodesic orbit inner product $(\cdot,\cdot)$ on $\m$ determines a super-adapted system of vectors $b^i=(b^i_1,b^i_2,\dots,b^i_m)\in \mathbb{R}^m, \; i=1,\dots,m-1$, via the decomposition \eqref{decomploort}.
\end{proposition}
\begin{proof}
For every eigenvalue of the metric endomorphism $A$, the corresponding eigenspace defines a self-saturated subspace in $\br^m$. In each of those subspaces, we choose an orthonormal basis as in Lemma~\ref{selfsatur2}. By Lemma~\ref{selfsatur2} and \eqref{basis_sadapt}, the union of these bases is a super-adapted system of vectors, as required.
\end{proof}

The following Theorem implies Theorem~\ref{th:gointro}.

\begin{theorem}\label{theoclass.1}
Let $F$ be a connected, compact, simple Lie group and let $g$ be an invariant metric on a Ledger--Obata space $F^m/\diag(F)$. The following statements are equivalent.

\begin{enumerate} [label=\emph{(\arabic*)},ref=\arabic*]
  \item \label{it:gothgo}
  The metric $g$ is geodesic orbit (so that the Ledger--Obata space is a geodesic orbit space).

  \item \label{it:gothnr}
  The metric $g$ is naturally reductive.

  \item \label{it:gothexplicit}
  The metric $g$ is defined by the inner product
      \begin{equation*}
        (\cdot,\cdot)=\gamma_1 \ip|_{\p_1}+\gamma_2 \ip|_{\p_2}+\dots +\gamma_{\m-1} \ip|_{\p_{m-1}},
    \end{equation*}
  where $\ip$ is minus the Killing form on $\g$, and $\p_i=\{(b^i_1,b^i_2,\dots,b^i_m)\otimes X\, | \, X\in \f\}$, $\gamma_i >0$, for $i=1,\dots,m-1$, are such that
  \begin{enumerate}[label=\emph{(\alph*)}, ref=\alph*]
    \item \label{it:gothexplicitsuper}
    the set $\{b^1, b^2, \dots, b^{m-1}\}\subset \mathbb{R}^m$ is a super-adapted system, and
    \item \label{it:gothexplicitCi}
    there exist $C_i \in \br, \; i=1,2,\dots,m-1$, such that
    \begin{equation*}
        \Big( 1 - \frac{\gamma_i}{\gamma_j} \Big) a_j^i =C_i, \quad \text{for } j \ne i,
    \end{equation*}
    where $b^i \diamond b^j =a^j_i b^i + a_j^i b^j$.
  \end{enumerate}
\end{enumerate}
\end{theorem}

\begin{proof} \eqref{it:gothgo} $\Rightarrow$ \eqref{it:gothexplicit}. Assertion \eqref{it:gothexplicit}\eqref{it:gothexplicitsuper} follows from Proposition~\ref{selfsatur3}. For assertion \eqref{it:gothexplicit}\eqref{it:gothexplicitCi}, we can assume that $m \ge 4$. Consider the metric endomorphism
\begin{equation*}
A =\gamma_1 \Id|_{\p_1}+\gamma_2 \Id|_{\p_2}+\dots +\gamma_{m-1} \Id|_{\p_{m-1}}\,
\end{equation*}
for $(\cdot,\cdot)$. By Lemma~\ref{GO-criterion}, $(\cdot,\cdot)$ is geodesic orbit if and only if for any $X \in \m$ there exists $Z\in \h$ such that
$[X+Z, A\,X]=0$. Take an arbitrary $X=X_1+X_2+\dots+X_{m-1} \in \m$, where $X_k \in\p_k$, and $X_k=(b^k_1 ,b^k_2,\dots,b^k_m) \otimes Z_k$ for some $Z_k \in \f$. Then $AX = \sum_{k=1}^{m-1} \gamma_k X_k$, and every $Z\in \h$ has the form $Z=(Z_0,Z_0,\dots,Z_0)$, where $Z_0 \in \f$.

We have
\begin{align*}
    [X, A\,X] &=\sum_{i,j} [X_i, \gamma_j X_j]=\sum_{i,j} \gamma_j (b^i \diamond b^j) [Z_i,Z_j]= \sum_{i,j} \gamma_j  (a^j_i  b^i + a_j^i  b^j)  [Z_i,Z_j] \\
    &=\sum_{i,j} \gamma_j  a^j_i  b^i   [Z_i,Z_j]+\sum_{i,j} \gamma_i   a_i^j  b^i  [Z_j,Z_i]= \sum_{i,j} (\gamma_j-\gamma_i)  a^j_i  b^i   [Z_i,Z_j]\,,\\
    [Z, A\,X]&=\sum_{i}\gamma_i b^i  [Z_0,Z_i].
\end{align*}
Since the vectors $b^1, \dots, b^{m-1}$ are linear independent, the equation $[X+Z, A\,X]=0$ is equivalent to the following system of equation.
\begin{equation}\label{gocondcoord}
\sum_{j=1}^{m-1} (\gamma_i-\gamma_j) a^j_i [Z_j,Z_i]+\gamma_i [Z_0,Z_i]=0,\quad i=1,\dots,m-1\,.
\end{equation}

Consider a subalgebra $\su(2)\subset \f$ spanned by vectors $E_1, E_2, E_3$ such that $[E_1,E_2]=E_3$, $[E_2,E_3]=E_1$, and $[E_3,E_1]=E_2$.
Fix pairwise distinct $i, j, k = 1, 2, \dots, m-1$, and take $Z_i=E_1, \; Z_j=E_2, \; Z_k=E_3$, and $Z_l=0$ for $l \not \in \{0,i,j,k\}$ in \eqref{gocondcoord}. We get
\begin{align*}
-(\gamma_i-\gamma_j)  a^j_i   E_3+(\gamma_i-\gamma_k)  a^k_i   E_2+\gamma_i  [Z_0,E_1]&=0,\\
-(\gamma_j-\gamma_k)  a^k_j   E_1+(\gamma_j-\gamma_i)  a^i_j   E_3+\gamma_j  [Z_0,E_2]&=0,\\
-(\gamma_k-\gamma_i)  a^i_k   E_2+(\gamma_k-\gamma_j)  a^j_k   E_1+\gamma_k  [Z_0,E_3]&=0.
\end{align*}
Then $Z_0$ is in the normaliser of $\su(2)$ in $\f$, and hence $Z_0=d_1 E_1+d_2 E_2+d_3 E_3+Z_0'$, where $Z_0'$
is in the centraliser of $\su(2)$ in $\f$. It follows that
\begin{equation*}
(\gamma_j-\gamma_i)  a^i_j=-\gamma_j  d_1\quad\mbox{ and }\quad (\gamma_k-\gamma_i)  a^i_k  =- \gamma_k   d_1\,,
\end{equation*}
which implies $(1 - \frac{\gamma_i}{\gamma_j}) a^i_j = (1 - \frac{\gamma_i}{\gamma_k} ) a^i_k$. Since $i,j,k$ are arbitrary, we get \eqref{it:gothexplicitCi}.

\smallskip

\eqref{it:gothexplicit} $\Rightarrow$ \eqref{it:gothgo}. Let $X \in \m$ be arbitrary. We have $X=X_1+X_2+\dots+X_{m-1}$, where $X_k=(b^k_1, b^k_2, \dots, b^k_m) Z_k \in \p_k$ and $Z_k \in \f$. Take $Z_0=-\sum_{j=1}^{m-1} C_j Z_j \in \f$. By \eqref{it:gothexplicitCi} we have $\sum_{j=1}^{m-1} C_j Z_j=\sum_{j=1}^{m-1} (1 - \frac{\gamma_j}{\gamma_i}) a^j_i Z_j$ for any $i=1,\dots,m-1$, which implies \eqref{gocondcoord}. Hence we obtain $[X+Z,A\,X]=0$ for $Z=(1,1,\dots, 1) \otimes Z_0$ and so the space is geodesic orbit by Lemma~\ref{GO-criterion}.

\smallskip

\eqref{it:gothnr} $\Leftrightarrow$ \eqref{it:gothgo}. The implication \eqref{it:gothnr} $\Rightarrow$ \eqref{it:gothgo} is always true. For the converse, note that the correspondence $X \mapsto Z$ such that $[X+Z,A\,X]=0$ can be chosen linear (and $\Ad(H)$-equivariant), as constructed above. Then the metric is naturally reductive by \cite[Proposition~2.10]{KV}.
\end{proof}

By Theorem~\ref{theoclass.1}, any geodesic orbit metric on a Ledger--Obata space is naturally reductive and as such, can be described using Theorem~\ref{th:natred}. An explicit construction taking $\g_m$ for an $\Ad(H)$-invariant complement to $\h$ is then given in Corollary~\ref{cor:natredgm}. We will continue working with the choice of $\m$ for an $\Ad(H)$-invariant complement to $\h$ and will clarify the dependence of the structure of the submodules $\p_i$ and the corresponding super-adapted system $\{b^i\}$ in Theorem~\ref{theoclass.1}\eqref{it:gothexplicit} of the eigenvalue structure of the metric endomorphism $A$.

Let $\m=\sum_{\alpha}\m_{\alpha}$ be the decomposition into the sum of the eigenspaces of the endomorphism $A$. Note that an eigenvalue $\alpha$ is simple if and only if $\dim (\m_{\alpha})=\dim (\f)$.

\begin{corollary}\label{corclass.1} If the metric $(\cdot,\cdot)$ is geodesic orbit and $C_i=0$ for some $i=1,2,\dots,m-1$, then
$[\p_i,\p_j]\subset \p_i$ for every $j \neq i$ such that $\gamma_i \neq \gamma_j$.
\end{corollary}

\begin{proof} For any $j \neq i$ we have $0=C_i=(1 - \frac{\gamma_i}{\gamma_j}) a_j^i$, and so $\gamma_i\neq \gamma_j$ implies
$a_j^i=0$, that is, $(b^i \diamond b^j) \f =(a^j_i b^i + a_j^i b^j) \f=a^j_i b^i \f \subset \p_i$.
\end{proof}

Note that the condition in Corollary \ref{corclass.1} is very restrictive if $\gamma_i$ is a simple eigenvalue of~$A$: the vector $b^i$ has exactly two non-zero coordinates. Indeed, assume that the first $r$ coordinates of $b^i$ are non-zero and the remaining $m-r$ are zeros (note that $r \ge 2$ by~\eqref{basis_adapt}). The fact that $b^i \diamond b^j$ is a multiple of $b^i$, for all $j \ne i$, implies that the first $r$ coordinates of each of $b^j$ are the same. But if $r \ge 3$, this implies that the vectors $b^j, \; j \ne i$, are linear dependent.

\begin{corollary}\label{corclass.2} If the metric $(\cdot,\cdot)$ is geodesic orbit and $\alpha$ is not a simple eigenvalue of $A$, then $[\m_{\alpha},\m_{\beta}]\subset \m_{\alpha}$ for every $\beta \neq \alpha$. In particular, if both $\alpha$ and $\beta$ are non-simple eigenvalues, then $[\m_{\alpha},\m_{\beta}]=0$.
\end{corollary}

\begin{proof}
Let $\p_i, \p_j \subset \m_{\alpha}, \; i \ne j$. Since $\gamma_i=\gamma_j$, we have $C_i=C_j=0$ by Theorem~\ref{theoclass.1}\eqref{it:gothexplicit}\eqref{it:gothexplicitCi}. Hence for any $\p_i\subset \m_{\alpha}$ and any $\p_k \not\subset \m_{\alpha}$, we get $0=C_i=( 1 - \frac{\gamma_i}{\gamma_k}) a_k^i$, and $\gamma_i\neq \gamma_k$ implies $a_k^i=0$, that is, $(b^i \diamond b^k) \f =(a^k_i b^i + a_k^i b^k) \f = a^k_i b^i \f \subset \p_i\subset \m_{\alpha}$.

It follows that $[\m_{\alpha},\m_{\beta}]\subset \m_{\alpha}$ for any $\beta \neq \alpha$.
\end{proof}

Note, that the number of non-zero coordinates of any vector $b^i$ with $\mathfrak{p_i}\subset \m_{\alpha}$ is at most the multiplicity of $\alpha$ plus $1$. This follows from the argument similar to the above: if the first $r$ coordinates of $b^i$ are non-zero and the rest are zeros, then the first $r$ coordinates of every vector $b^j$ with $\gamma_j \ne \alpha$, are the same; as these vectors are linear independent, there are no more than $m-r+1$ of them.

From this we easily get the following.

\begin{corollary}\label{corclass.3} If the metric $(\cdot,\cdot)$ is geodesic orbit and the corresponding super-adapted system $\{b^1,\dots, b^{m-1}\}$ is such that $b^i_k \ne 0$, for all $i, k$, then all the eigenvalues of $A$ are simple and $C_i\neq 0$ for all $i$.
\end{corollary}

We give an example of a super-adapted system and the corresponding geodesic orbit inner product, as in the statement of Corollary~\ref{corclass.3}.
\begin{example}\label{exasas.1}
Take arbitrary positive real numbers $z_1 < z_2 < \dots < z_m$ and consider the function $\varphi(t):= \sum_{k=1}^m \frac{z_k}{z_k-t}$. The equation $\varphi(t)=0$ has exactly $m-1$ distinct roots and all these roots are positive (there is exactly one root $t_i$ on the interval $(z_i, z_{i+1})$, for $i=1, \dots, m-1$). For $i=1, \dots, m-1, \; k=1, \dots, m$, define $b^i_k= \mu_i\frac{z_k}{z_k-t_i}$, where $\mu_i > 0$ are chosen in such a way that $\sum_{k=1}^m (b^i_k)^2=1$. Then $\sum_{k=1}^m b^i_k=0$, and for $i \ne j$ we have $b^i \diamond b^j =\frac{\mu_j t_i}{t_i-t_j} b^i+ \frac{\mu_i t_j}{t_j-t_i} b^j$ which implies $\sum_{k=1}^m b^i_k b^j_k=0$. Therefore the system $\{b^i\}$ is super-adapted.

Furthermore, by Theorem~\ref{theoclass.1}\eqref{it:gothexplicit}\eqref{it:gothexplicitCi} we get $C_i = (1 - \frac{\gamma_i}{\gamma_j}) \frac{\mu_i t_{j}}{t_{j}-t_{i}}$ for $i \ne j$. Then $C_i \neq 0$ and so $\gamma_i\neq \gamma_j$. Denote $\rho_i=\frac{C_i}{\mu_i \gamma_i t_i} \ne 0$. Then $\rho_i (t_i^{-1}-t_j^{-1}) = \gamma_i^{-1}-\gamma_j^{-1}$, for all $i \ne j$, and so $\rho_i=\rho_j$. Then $\rho_i = \rho \ne 0$ and $\gamma_i^{-1}= \rho t_i^{-1} + \lambda$ for some $\lambda \in \br$. Therefore, $\gamma_i=\frac{t_i}{\rho + \lambda t_i}$ for all $i=1,\dots,m-1$.
Hence we get a two-parameter family of geodesic orbit metrics for any such super-adapted system.
\end{example}

\section{Geodesic orbit Ledger--Obata manifolds} 
\label{s:LOm}

In this section, we classify the Ledger--Obata spaces which are geodesic orbit manifolds: any geodesic of such space is the orbit
of a one-parameter subgroup of the full isometry group (or equivalently, of its connected identity component).
\vspace{10mm}

\subsection{Proof of Theorem~\ref{th:red} and Corollary~\ref{cor:gom}}

\label{ss:LOmth}
\begin{proof}[Proof of Theorem~\ref{th:red}]
The space $F^m/\diag(F)$ with an invariant metric is isometric to the Lie group $F^{m-1}$ with a left-invariant metric which is also right-invariant with respect to the subgroup $\diag(F)$. 

{
We need the following variation of \cite[Theorem~3]{Ozeki}. Suppose a compact, connected and simply connected Lie group $K$ admits a left-invariant metric $\rho$ that is $\Ad(Q)$-invariant with respect to a connected subgroup $Q\subset K$. This means that $\rho$ is right-invariant with respect to $K$ or equivalently, that the inner product $(\cdot,\cdot)$ on the Lie algebra $\kg=\Lie(K)$ which generates $\rho$ satisfies $([U,X],Y)+(X,[U,Y])=0$ for any $X,Y \in \kg$ and any $U \in \mathfrak{q} = \Lie(Q)$ (so that $(\cdot,\cdot)$ is $\ad(\mathfrak{q})$-invariant). In general, $K$ does not have to be normal in the full connected isometry group $G$ of the metric $\rho$, but we have the following.

\begin{proposition}\label{groupchange}
There is a normal subgroup $K_1$ in $G$ isomorphic to $K$ that admits a left-invariant and $\Ad(Q_1)$-invariant metric $\rho_1$ isometric to $\rho$,
where $Q_1$ is the image of $Q$ under the isomorphism between $K$ and $K_1$.
\end{proposition}

\begin{proof} Let $H$ be the isotropy group of $\rho$ at the identity. At the level of Lie algebras, we have a decomposition $\g=\kg \oplus \h$ into the direct sum of linear subspaces. Then by \cite[Theorem~1]{Ozeki}, there exists a decomposition $\g=\kg_1\oplus\h_1$ into the direct sum of Lie algebras, where $\kg_1$ is isomorphic to $\kg$ and $\h_1$ is isomorphic to $\h$, and moreover, there is a Lie algebra homomorphism $\varphi:\kg_1 \rightarrow \h_1$ such that $\kg = \{X+\varphi(X) \,|\,X\in \kg_1\}$. Consider a subalgebra $\mathfrak{q}_1 \subset \kg_1$ such that $\mathfrak{q}=\{X+\varphi(X)\,|\,X\in \mathfrak{q}_1\}$ and
the inner product $(\cdot,\cdot)_1$ on $\kg_1$ such that $(X,Y)_1=(X+\varphi(X),Y+\varphi(Y))$ for $X,Y\in \kg_1$. This inner product is isometric to $(\cdot,\cdot)$.

To prove that $(\cdot,\cdot)_1$ is $\ad(\mathfrak{q}_1)$-invariant, take arbitrary $X \in \kg_1$ and $U\in \mathfrak{q}_1$. Then
\begin{align*}
([U,X],X)_1 &=([U,X]+[\varphi(U),\varphi(X)], X+\varphi(X)) \\
&=([U+\varphi(U),X+\varphi(X)],X+\varphi(X))=0,
\end{align*}
since $(\cdot,\cdot)$ is $\ad(\mathfrak{q})$-invariant, $\varphi$ is a homomorphism and $[\kg_1, \h_1]=0$.

The rest of proof is the same as the proof of \cite[Theorem~3]{Ozeki}.
\end{proof}

\begin{remark} \label{rem:nonsimplycon}
If we drop the condition that $K$ is simply connected, then $K_1$ is locally isomorphic to $K$ and acts on $K$ transitively.
\end{remark}
}

Let $G$ be the connected isometry group of $F^m/\diag(F)$. By Proposition~\ref{groupchange}, there is an isometric left-invariant metric on $F^{m-1}$ which is also right-invariant with respect to the subgroup $\diag(F)$ and is such that the group of left translations is normal in $G$. Then by \cite[Theorem~3]{OT} the group $G$ is the direct product of the group $F^{m-1}$ of left translations and a certain subgroup $K \subset F^{m-1}$ of right translations. We know that $K \supset \diag(F)$, and so by Lemma~\ref{l:subal}, $K=\prod_{i=1}^s K_i$, where $K_i = \diag(F) \subset F^{m_i-1}$, $m_i >1$ and $\prod_{i=1}^s F^{m_i-1} =F^{m-1}$. The tangent spaces to the subgroups $F^{m_i-1}$ are orthogonal and so the given Ledger--Obata space $F^m/\diag(F)$ is isometric to the product of Ledger--Obata spaces $F^{m_i}/\diag(F)$, and its connected isometry group $G$ is the product $\prod_{i=1}^m F^{m_i}=F^{m+s-1}$. Note that each of the factors $F^{m_i}/\diag(F)$ is irreducible by \cite[Proposition~8]{Nomizu}.
\end{proof}

\begin{proof}[Proof of Corollary~\ref{cor:gom}]
A smooth curve in a product manifold is a geodesic parameterised by an affine parameter if and only if its projections to the factors are. As the connected isometry group of the Ledger--Obata space is the product of the connected isometry groups of its factors, the claim follows from Theorem~\ref{th:gointro}. 
\end{proof}

\subsection{Reducibility of Ledger--Obata spaces}
\label{ss:reduc}

To make Corollary~\ref{cor:gom} more effective in practice we present in this section a method of determining that a given inner product generates a reducible invariant metric on a Ledger--Obata space. The method is based on the holonomy computation.

Suppose a Ledger--Obata space is reducible. Let $\m$ be the $\ip$-orthogonal complement to $\h \subset \g$ as defined by \eqref{eq:defm} and let $A$ be the metric endomorphism on $\m$. Extend $A$ to the symmetric endomorphism $C$ on $\g$ as in Section~\ref{ss:holo}. Let $\m=\m_1 \oplus \m_2$ be the $(\cdot, \cdot)$-orthogonal decomposition of $\m$ into the subspaces invariant relative to the holonomy algebra. Denote $m_a=\dim \m_a/\dim \f + 1 > 1$; then $m_1 + m_2 = m+1$. By Lemma~\ref{l:holinv}, both subspaces $\g_a=\h \oplus \m_a$ are subalgebras in $\g$, and so by Lemma~\ref{l:subal}, there exist two partitions $\P^a = \{S^a_i\}_{i=1}^{m_a}, \; a=1,2$, of the set $\{1, 2, \dots, m\}$ such that $\g_a= \Span(\chi^a_i \otimes X \, | \, X \in \f, \, i=1, 2, \dots, m_a)$, where for $a=1, 2, \; i=1, 2, \dots, m_a$, the $k$-th component of the vector $\chi^a_i \in \br^m$ is $1$, if $k \in S^a_i$ and is zero otherwise. 

By construction, $\g=\g_1+\g_2, \; \dim \g_a = (m_a+1) \dim \f$ and $\g_1 \cap \g_2 = \h$ which is equivalent to the fact that the partitions $\P^a$ satisfy the following conditions (note that \eqref{it:part3} follows from the other two):
\begin{enumerate}[label=(\alph*), ref=\alph*]
  \item \label{it:part1}
  $|\P^a| = m_a > 1$ and $|\P^1|+|\P^2|=m+1$;

  \item \label{it:part2}
  If for nonempty sets $J_a, \; a=1,2$, we have $\cup_{i \in J_1} S^1_i = \cup_{i \in J_2} S^2_i$, then $J_a=\{1, 2, \dots, m_a\}$.

  \item \label{it:part3}
  for any $i_1 \le m_1, \; i_2 \le m_2$, the parts $S^1_{i_1}$ and $S^2_{i_2}$ have no more than one element in common.
\end{enumerate}
All such pairs of partitions can be constructed using the following simple algorithm.

\begin{algorithm*}
Given $m > 2$, consider a tree on $m+1$ vertices other than the complete bipartite graph $K_{1,m}$. Label the edges arbitrarily by the numbers $1, 2, \dots, m$. Choose an arbitrary vertex and colour it white, then colour all its adjacent vertices black, then colour all their adjacent vertices white, and so on. Let $m_1$ and $m_2$ be the number of white and black vertices respectively. Label the white vertices by the numbers $1, 2, \dots, m_1$ and the black vertices, $1, 2, \dots, m_2$. Then the parts of the partition $\P^1$ are $S^1_i, \; i=1, \dots, m_1$, where $S^a_i$ is the set of labels on the edges incident to the white vertex labelled $i$, and similarly for $\P^2$ and the black vertices.
\end{algorithm*}
It is a matter of simple verification to see that there is a bijection, up to relabelling, between the pairs of partitions satisfying \eqref{it:part1}, \eqref{it:part2} and \eqref{it:part3}, and isomorphic classes of trees on $m+1$ vertices.

We next consider the endomorphism $C$ on $\g$. From Section~\ref{ss:holo} we know that $C$ is $\ad(\h)$-invariant, symmetric, positive semidefinite, and $\ker C = \h$. It follows that there is a unique $m \times m$ matrix $T$ such that for every $a \in \br^m$ and $X \in \f$ we have $C(aX)=(Ta) X$. The matrix $T$ is symmetric, positive semidefinite, and $\ker T = (1,1, \dots, 1)$.

We want to characterize those such matrices $T$ which produce an endomorphism $C$, for which the holonomy algebra given in Lemma~\ref{reduce-criterion} is reducible, with $\m=\m_1 \oplus \m_2$ being the decomposition into invariant subspaces. We say that an $m \times m$ matrix $N$ \emph{agrees} with a partition $\P$ of the set $\{1,2, \dots, m\}$ if $N_{ij}=0$ whenever $i$ and $j$ belong to the different parts of $\P$.

\begin{proposition} \label{p:matrix}
Let $T$ be an $m \times m$ symmetric, positive semidefinite matrix with $\ker T = (1,1, \dots, 1)$. The endomorphism $C$ of $\g$ corresponding to $T$ defines the inner product whose holonomy algebra is reducible, with complementary invariant subspaces $\m_1, \m_2 \subset \m$ if and only if $T = T^{(1)} +
T^{(2)}$, where $T^{(a)}, \; a=1, 2$, is an $m \times m$ symmetric matrix with $\ker T^{(a)} \ni (1,1, \dots, 1)$ which agrees with the partition $\P^a$.
\end{proposition}
\begin{proof}
From Lemma~\ref{reduce-criterion} and the following discussion in Section~\ref{ss:holo} we obtain that the holonomy algebra defined by $C$ is reducible, with complementary invariant $(\cdot, \cdot)$-orthogonal subspaces $\m_1, \m_2 \subset \m$ if and only if $C$ is $\ad(\f)$-invariant and for every $Z \in \m$ and $X \in \m_1, \; Y \in \m_2$ we have
\begin{align*}
  0&=(\Gamma_Z X,Y) =\< C (\Gamma_Z X),Y \>=\< C[Z,X],Y \>+\< [Z,CX],Y\> -\< [CZ,X],Y\>\\
  &=\<[X,CY] + [CX,Y] - C[X,Y],Z\>,
\end{align*}
so that $C$ acts ``like a derivation": for all $X \in \m_1, \; Y \in \m_2$, we have $C[X,Y]=[X,CY] + [CX,Y]$. As $C$ is symmetric and $C|_\h=0$, this is equivalent to the same equation being satisfied for all $X \in \g_1=\m_1 \oplus \h, \; Y \in \g_2=\m_2 \oplus \h$. In terms of the matrix $T$, this equation is equivalent to the equation
\begin{equation}\label{eq:der}
  T(x \diamond y) = (Tx) \diamond y + x \diamond (Ty),
\end{equation}
for all $x \in \Span(\chi^1_1, \chi^1_2, \dots, \chi^1_{m_1}), \;  y \in \Span(\chi^2_1, \chi^2_2, \dots, \chi^2_{m_2})$. For every $i \in \{1,2, \dots, m\}$, there are uniquely defined $1 \le r_1 \le m_1, \; 1 \le r_2 \le m_2$ such that $i \in S^1_{r_1}, S^2_{r_2}$; then by property~\eqref{it:part3}, $S^1_{r_1} \cap S^2_{r_2} = \{i\}$ and so $\chi^1_{r_1} \diamond \chi^2_{r_2} = e_i$, the $i$-th vector of the standard basis for $\br^m$. Substituting $x = \chi^1_{r_1}, \; y=\chi^2_{r_2}$ in \eqref{eq:der} we obtain $\sum_{j=1}^{m} T_{ij} e_j = \sum_{k \in S^1_{r_1}, l \in S^2_{r_2}} T_{kl} (e_k+e_l)$. It follows that $T_{ij}=0$ if $j \notin S^1_{r_1} \cup S^2_{r_2}$. In other words, if $T_{ij} \ne 0$, then $i$ and $j$ belong either to the same part of $\P^1$, or to the same part of $\P^2$. This property is in fact equivalent to \eqref{eq:der}. Indeed, suppose it holds for the matrix $T$. Take arbitrary $r_a=1, 2, \dots, m_a, \; a=1,2$ and substitute $x = \chi^1_{r_1}, \; y=\chi^2_{r_2}$ in \eqref{eq:der}. If $S^1_{r_1} \cap S^2_{r_2} = 0$, then $\chi^1_{r_1} \diamond \chi^2_{r_2} = 0$ and no $i \in S^1_{r_1}, \; j \in S^2_{r_2}$ belong to the same part of any of either $\P^1$ or $\P^2$; hence both sides of \eqref{eq:der} are zeros. If $S^1_{r_1} \cap S^2_{r_2} = \{i\}$, then $T_{ij}=0$ for all $j \notin S^1_{r_1} \cup S^2_{r_2}$, and $T_{kl}=0$ for all $k \in S^1_{r_1} \setminus \{i\}, \; l \in S^2_{r_2} \setminus \{i\}$, and again, the sides of \eqref{eq:der} are equal.

Now for $i \ne j$, define $T^{(a)}_{ij}$ to be equal to $T_{ij}$ if $i$ and $j$ belong to the same part of the partition $\P^a, \; a =1,2$, and to be zero otherwise. Define the diagonal elements $T^{(a)}_{ii}$ in such a way that $\ker T^{(a)} \ni (1,1, \dots, 1)$. Then the matrix $T^{(a)}$ agrees with the partition $\P^a$. Furthermore, as $i \ne j$ can not belong to the same parts of the both partitions $\P^1$ and $\P^2$ (property~\eqref{it:part3}), the matrices $T$ and $T^{(1)} + T^{(2)}$ agree outside the diagonals. This implies that $T = T^{(1)} + T^{(2)}$, as the kernel of the both sides contains the vector $(1,1, \dots, 1)$.
\end{proof}

\begin{example} \label{ex:graph}
Let $m=7$. Consider the tree on $m+1=8$ vertices as on the left in Figure~\ref{fig:tree}. Label the vertices and edges as on the right in Figure~\ref{fig:tree}:
white vertices $1$ through to $3$, black vertices $1$ through to $5$ and the edges $1$ through to $7$.

\begin{figure}[t]
\includegraphics{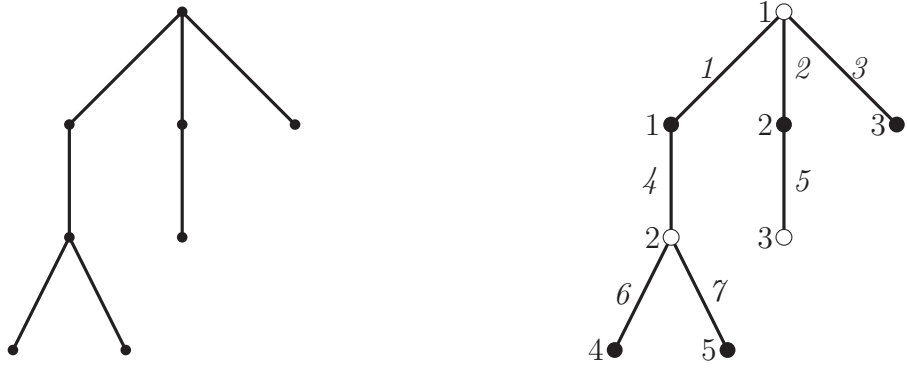}
\caption{A tree and its labelling which produces the pair of partitions.}\label{fig:tree}
\end{figure}

By our algorithm, the tree defines two partitions of the set $\{1,2, \dots, 7\}$ given by $\P^1= 123 | 467 | 5$ and $\P^2= 14 | 25 | 3 | 6 | 7$.
This pair of partitions defines two subalgebras $\g_1 = \Span(\chi^1_1, \chi^1_2, \chi^1_3) \otimes \f$ and
$\g_2 = \Span(\chi^2_1, \chi^2_2, \chi^2_3, \chi^2_4, \chi^2_5) \otimes \f$ of the algebra $\g=7 \f$,
where $\chi^1_1 = e_1+e_2+e_3, \; \chi^1_2 = e_4+e+6+e_7, \; \chi^1_3 = e_5$ and $\chi^2_1 = e_1+e_4, \;
\chi^2_2 = e_2+e_5, \; \chi^2_3 = e_3, \; \chi^2_4 = e_6, \; \chi^2_5 = e_7$. We have $\g_1 \simeq 3 \f, \; \g_2 \simeq 5 \f, \; \g_1 + \g_2 = \g$ and
$\g_1 \cap \g_2 = \h = \diag(\f)$. This defines the splitting $\m = \m_1 \oplus \m_2$, where $\m, \m_1$ and $\m_2$ are the
$\ip$-orthogonal complements to $\h$ in $\g, \g_1$ and $\g_2$ respectively. Then the matrix $T$ which defines the inner product on $\m$ is given by
$T = T^{(1)} + T^{(2)}$, where $T^{(a)}$ agrees with $\P^a$ for $a=1,2$:

\begin{eqnarray*}
  T = T^{(1)} + T^{(2)} &= \left(
                            \begin{array}{ccccccc}
                              x_1 + x_2 & -x_1 & -x_2 & 0 & 0 & 0 & 0 \\
                              -x_1 & x_1 + x_3 & -x_3 & 0 & 0 & 0 & 0 \\
                              -x_2 & -x_3 & x_2 + x_3 & 0 & 0 & 0 & 0 \\
                              0 & 0 & 0 & x_4 + x_5 & 0 & -x_4 & -x_5 \\
                              0 & 0 & 0 & 0 & 0 & 0 & 0 \\
                              0 & 0 & 0 & -x_4 & 0 & x_4 + x_6 & -x_6 \\
                              0 & 0 & 0 & -x_5 & 0 & -x_6 & x_5 + x_6 \\
                            \end{array}
                          \right) \\
                          &+
                          \left(
                            \begin{array}{ccccccc}
                              y_1 & 0 & 0 & -y_1 & 0 & 0 & 0 \\
                              0 & y_2 & 0 & 0 & -y_2 & 0 & 0 \\
                              0 & 0 & 0 & 0 & 0 & 0 & 0 \\
                              -y_1 & 0 & 0 & y_1 & 0 & 0 & 0 \\
                              0 & -y_2 & 0 & 0 & y_2 & 0 & 0 \\
                              0 & 0 & 0 & 0 & 0 & 0 & 0 \\
                              0 & 0 & 0 & 0 & 0 & 0 & 0 \\
                            \end{array}
                          \right),
\end{eqnarray*}
where $x_1, x_2, x_3, x_4, x_5, x_6, y_1, y_2 \in \br$ are such that $T$ is positive semidefinite, of rank $6$; for example, all positive. The metric on the Ledger--Obata space $F^7/\diag(F)$ generated by the inner product defined by $T$ is reducible: $F^7/\diag(F) = (F^3/\diag(F)) \times (F^5/\diag(F))$, and the tangent spaces to the factors are $\m_1$ and $\m_2$ respectively.

To construct all reducible metrics on $F^7/\diag(F)$ one has to similarly analyse all the non-isomorphic trees with $8$ vertices other than $K_{1,7}$.
\end{example}

\vspace{10mm}

\end{document}